\documentclass[12pt]{article}
\usepackage{float}
\usepackage{makeidx}
\usepackage{latexsym}

\usepackage[margin=2.5cm]{geometry}
\usepackage[dvips]{graphicx}
\usepackage[english]{babel}
\usepackage{amssymb}
\usepackage{amsmath}
\usepackage{amsthm}
\usepackage{lineno}
\usepackage{xcolor}
\definecolor{blau}{rgb}{0.1,0.0,0.9}
\definecolor{gruen}{cmyk}{1.0,0.2,0.7,0.07}
\definecolor{mag}{cmyk}{0.0,0.9,0.3,0.0}

\newtheorem{theorem}{Theorem}[section]
\newtheorem{lemma}[theorem]{Lemma}
\newtheorem{corollary}[theorem]{Corollary}

\theoremstyle{definition}

\begin{document}
\date{\today}
\title{On restricted colorings of $(d,s)$-edge colorable graphs}

\author{
{\sl Lan Anh Pham }\footnote{Department of Mathematics, 
Ume\aa\enskip University, 
SE-901 87 Ume\aa, Sweden.
{\it E-mail address:} lan.pham@umu.se}
}
\maketitle

\bigskip
\noindent
{\bf Abstract.}
A cycle is \emph{$2$-colored} if its edges are properly colored by two distinct colors.
A \emph{$(d,s)$-edge colorable graph} $G$ is a $d$-regular graph that admits a proper $d$-edge coloring
in which every edge of $G$ is in at least $s-1$ $2$-colored $4$-cycles. Given a $(d,s)$-edge colorable graph 
$G$ and a list assigment $L$ of forbidden colors for the edges of $G$ satisfying certain sparsity 
conditions, we prove that there is a proper $d$-edge coloring of $G$ 
that avoids $L$, that is, a proper edge coloring $\varphi$ of $G$
such that $\varphi(e) \notin L(e)$ for every edge $e$ of $G$.

\bigskip

\noindent

\section{Introduction}
A graph $G$ is \emph{$k$-edge list colorable} or \emph{$k$-edge choosable} if for every assignment 
of lists of at least $k$ colors to the edges of $G$, there is a proper edge coloring of $G$ using only colors from the lists.
The \emph{list chromatic index} or \emph{edge choosability} $\chi'_l(G)$ of a graph $G$ is the minimum number $k$
such that $G$ is $k$-edge list colorable. The most famous conjecture about list coloring states that $\chi'_l(G)=\chi'(G)$ \cite{Jensen},
where $\chi'(G)$ is the chromatic index of $G$, refering to the smallest number of colors needed to color the edges of $G$ 
to obtain a proper edge coloring. In $1994$, Galvin \cite{Gal95} proved this conjecture for bipartite multigraphs, his result 
also answers a question of Dinitz in 1979 about the generalization of Latin squares which can be formulated as a result of 
list edge coloring of the complete bipartite graph 
$K_{d,d}$, that $\chi'_l(K_{d,d})=d$.

Meanwhile, H\"aggkvist \cite{Haggkvist} worked with sparser lists, his conjecture of avoiding arrays can be rewritten in
the language of graph theory to state that
there exists a fixed $0<\beta \leq \frac{1}{3}$
such that if each edge $e$ of $K_{d,d}$ is
assigned a list $L(e)$ of at most $\beta d$ colors from 
$\{1, \dots,d\}$ and at every vertex $v$ each color is forbidden on at 
most $\beta d$ edges adjacent to $v$, then there is a proper $d$-edge coloring $\varphi$ of $K_{d,d}$ that {\em avoids} the lists, 
i.e $\varphi(e) \notin L(e)$ for every edge $e$ of $K_{d,d}$; if such a coloring exists, then $L$ is \textit{avoidable}.
For the case when $d$ is a power of two, 
Andr\'en proved that such a $\beta$ exists
in \cite{Lina}; the full conjecture was later settled in the affirmative
in \cite{AndrenCasselgrenOhman}.

Casselgren et al. \cite{JohanQd} demonstrated that a similar result
holds for the family of hypercube graphs. A benefit of working with the complete bipartite graph $K_{d,d}$ ($d=2^t$, $t \in \mathbb{N}$) and 
the $d$-dimensional hypercube graph $Q_d$ ($d \in \mathbb{N}$) is that they are both regular graphs that have proper edge colorings in which every edge is in $(d-1)$ $2$-colored $4$-cycles.
The purpose of this paper is to study this type of problem for 
regular graphs where the number of $2$-colored $4$-cycles each edge 
is contained in can be smaller.

To be more specific, we consider the family of 
{\em $(d,s)$-edge colorable graphs}: 
a $d$-regular graph $G$ is called {\em $(d,s)$-edge colorable} 
if it admits a proper $d$-edge coloring $h$
in which every edge of $G$ is contained 
in at least $(s-1)$ $2$-colored $4$-cycles.
Note that since the number of $2$-colored $4$-cycles containing 
an edge $e$ of $G$ is at most $(d-1)$, 
$s$ can not exceed $d$. The $d$-edge coloring $h$ is called \emph{standard coloring}.
A \emph{standard matching} $M$ of $G$ is a maximum set of edges of $G$ 
all of which have the same color in the standard coloring $h$. 

The {\em distance} between two
edges $e$ and $e'$ is the number of edges in a 
shortest path between an endpoint of $e$ and an endpoint of $e'$;
a {\em distance-$t$ matching} is a matching where any
two edges are at distance at least $t$ from each other.
The {\em $t$-neighborhood} of an edge $e$ is
the graph induced by all edges of distance at most $t$ from $e$.

Throughout, we shall assume that the standard coloring 
$h$ for the edges of $G$ uses the set of colors $\{1,\dots, d\}$. 
Next, similarly to \cite{JohanQd}, using the colors $\{1,\dots, d\}$
we define a list assignment for the edges of a $(d,s)$-edge colorable graph.

	A list assignment $L$ for a $(d,s)$-edge colorable graph $G$ is {\em $\beta$-sparse}
	if the list of each edge is a (possibly empty) subset of $\{1,\dots,d\}$, 
	and
	\begin{itemize}
	
	\item[(i)] $|L(e)| \leq \beta s$ for each edge $e \in G$;
	
	\item[(ii)] for every vertex $v \in V(G)$,
	each color in $\{1,\dots,d\}$ occurs in at most $\beta s$
	lists of edges incident to $v$;
	
	\item[(iii)] for every $6$-neighborhood $W$, and every 
	standard matching $M$, any color appears at most $\beta s$
	times in lists of edges of $M$ contained in $W$.
        \end{itemize}

We can now formulate our main result.
\begin{theorem}
\label{maintheorem}
Let $G$ be a $(d,s)$-edge colorable graph of order $n$ 
and $L$ be a $\beta$-sparse list assignment for $G$.
If $s \geq 11$ and $\beta \leq 2^{-11}sd^{-1} (2n)^{-2^9ds^{-2}}$, 
then there is a proper $d$-edge coloring of $G$ which avoids $L$.
\end{theorem}

The following corollary gives us a shortened form of the condition on 
$\beta$ as a function of $d$ and $n$ when $s/d$ is constant.

\begin{corollary}
\label{shortenbeta}
Let $G$ be a $(d,s)$-edge colorable graph of order $n$ such that 
$s \geq 11$ and $s/d$ is constant; then there are positive constants 
$c_1,c_2$ such that if $\beta \leq c_1 (2n)^{-c_2d^{-1}}$, 
then any $\beta$-sparse list assignment $L$ for $G$ is avoidable.
\end{corollary}

If the value of $d$ in Corollary \ref{shortenbeta} is at least $c\log n$ 
($c$  constant), then $\beta$ will be constant.
\begin{corollary}
\label{constantbeta}
Let $G$ is a $(d,s)$-edge colorable graph of order $n$
such that $s \geq 11$ and suppose that there 
are constants $\kappa \leq 1$ and $c$ satisfying $s=\kappa d$ 
and $d \geq c \log n$; then there is a constant $\beta >0$ such that
any $\beta$-sparse list assignment $L$ for $G$ is avoidable.
\end{corollary}

Note that the complete bipartite graph $K_{d,d}$ ($d=2^t$, $t \in \mathbb{N}$) and
the $d$-dimensional hypercube graph $Q_d$ ($d \in \mathbb{N}$) are both $(d,d)$-edge colorable graphs
satisfying the condition in Corollary \ref{constantbeta}. Thus this corollary
generalizes the results in \cite{Lina, JohanQd, AndrenCasselgrenOhman}.
The next corollary examines the condition on $d$ and $s$
so that for every $(d,s)$-edge colorable graph $G$ of order $n$ 
and any $\beta$-sparse list assignment $L$ for $G$ satisfying that the length 
of every list in $L$ is constant, 
there is a proper $d$-edge coloring of $G$ which avoids $L$.

\begin{corollary}
\label{constantlist1}
Let $G$ be a $(d,s)$-edge colorable graph $G$ of order $n$. If $s \geq 11$ 
and for some constant $c$ we have 
$\dfrac{1}{2} (2^{-11}c^{-1} s^2d^{-1})^{2^{-9}d^{-1}s^{2}} \geq n$,
then any $\dfrac{c}{s}$-sparse list assignment $L$ for $G$ is avoidable.
\end{corollary}

If the length of every list in $L$ is bounded by a power of $s$, 
we have a slightly different condition on $d$ and $s$.
\begin{corollary}
\label{constantlist2}
Let $G$ be a $(d,s)$-edge colorable graph $G$ of order $n$. If $s \geq 11$ 
and for some constant $c$ we have $\dfrac{1}{2} (2^{-11}s^{2-c}d^{-1})^{2^{-9}d^{-1}s^{2}} \geq n$,
then any $s^{c-1}$-sparse list assignment $L$ for $G$ is avoidable.
\end{corollary}

Consider an arbitrary list assignment $L'$ (not necessariliy $\beta$-sparse)
for a $(d,s)$-edge colorable graph. In general, it is difficult to determine if $L'$ is avoidable or not. 
However, if the edges with forbidden lists are placed on a distance-$3$ matching, 
our method in fact immediately yields the following.
\begin{theorem}
\label{secondth}
	Let $L'$ be a list assignment for the edges of a $(d,s)$-edge colorable
	graph $G$ such that for each edge $e$ of $G$, $L'(e)\leq s-1$.
	If every edge $e$ satisfying $L'(e) \neq \emptyset$ belongs
	to a distance-$3$ matching in $G$, then $L'$ is avoidable.
\end{theorem}

The rest of the paper is organized as follows. In Section 2, after introducing some terminology and notation,
we prove Theorem \ref{maintheorem}. Our proof 
relies heavily on the fact that every edge is contained in a large 
number of $2$-colored $4$-cycles. It would be interesting to investigate 
if a similar result holds for graphs containing a certain amount of
$2$-colored $2c$-cycles ($c \in \mathbb{N}$, $c >2$). Section 3 gives some examples of classes of graphs that belong
to the family of $(d,s)$-edge colorable graphs. 

\section{Proof the main theorem}
Given a $(d,s)$-edge colorable graph $G$ of order $n$, let $h$ be a standard coloring of $G$.
For a vertex $u \in G$,
we denote by $E_u$ the set of edges with one endpoint being $u$,
and for a (partial) edge coloring $f$ of $G$, let $f(u)$
denote the set of colors on edges in $E_u$ under $f$.
If two edges $uv$ and $zt$ of $G$ are in a $2$-colored $4$-cycle
in $G$ 
then the edges $uv$ and $zt$ are {\em parallel}.

Given a proper coloring $h'$ of the edges of $G$, for an edge $e \in G$, 
any edge $e' \in G$ ($e' \neq e$) belongs to at most 
one $2$-colored $4$-cycle containing $e$. This property is obvious if $e'$ 
and $e$ are not adjacent; in the case when they have the same endpoint $u$, 
assume $e=uv$, $e'=uv'$ and $uvv_1v'u$ and $uvv_2v'u$
are two $2$-colored $4$-cycles containing $e$ and $e'$; 
then $h'(vv_1)=h'(vv_2)=h'(uv')$,
a contradiction since $vv_1$ and $vv_2$ are adjacent. 

Consider a $\beta$-sparse list assignment $L$ for $G$ and
a proper edge coloring $\phi$ of $G$.
An edge $e$ of $G$ is called a \textit{conflict edge}
(\textit{of $\phi$ with respect to $L$})  if $\phi(e) \in L(e)$. 
An \textit{allowed cycle} 
(\textit{under $\phi$ with respect to $L$}) of $G$ is a $4$-cycle
$\mathcal{C}=uvztu$ in $G$ that is $2$-colored under $\phi$, 
and such that interchanging colors on $\mathcal{C}$ yields a proper 
$d$-edge coloring $\phi'$ of $G$ where none of $uv$, $vz$, 
$zt$, $tu$ is a conflict edge. We call such an interchange {\em a swap in $\phi$}.

We shall establish that our main theorem holds by proving two lemmas.
In the following, $G$ is a $(d,s)$-edge colorable graph of order $n$,
$L$ is a $\beta$-sparse list assignment for $G$, and
$h$ is a standard coloring of $G$.
For simplicity of notation, we shall omit floor and ceiling signs whenever these are not crucial.
\begin{lemma}
\label{lemma1}
Let $0<\gamma, \tau<1$ be parameters such that $\beta \leq \gamma$ and
         \begin{equation}
	\label{e1}
	n\big(\dfrac{e \beta}{\gamma}\big)^{\gamma s} 
 	+\dfrac{nd}{2}\big(\dfrac{2e \beta}{\tau-2\beta}\big)^{(\tau-2\beta) s} < 1
	 \end{equation}
        There is a permutation $\rho$ of $\{1,\dots,d\}$, such that applying $\rho$ to
	the set of colors $\{1,\dots,d\}$ used in $h$, we obtain a proper $d$-edge coloring 
	$h'$ of $G$ satisfying the following:
        \begin{itemize}
	\item[(a)]  For every $6$-neighborhood $W$, and every standard matching $M$, 
	at most $\gamma s$ edges of $M \cap E(W)$ are conflict.
		
	\item[(b)] No vertex $u$ in $G$ satisfies that $E_u$ contains 
	more than $\gamma s$ conflict edges. 
	
	\item[(c)] Each edge in $G$ belongs to at least $(1-\tau)s$ 
	allowed cycles.
\end{itemize}
\end{lemma}
\begin{proof}
	Let $A$, $B$, $C$ be the number of permutations which do not 
	fulfill the conditions $(a)$, $(b)$, $(c)$, respectively. 
	Let $X$ be the number of permutations satisfying the three conditions 
	$(a)$, $(b)$, $(c)$. There are $d!$ ways to permute the colors, so we have
	$$X \geq d! - A- B - C$$
	We will now prove that $X$ is greater than $0$.
	
	\begin{itemize}
	\item Since all edges that are in the same standard matching
	have the same color under $h$ and 
	for every $6$-neighborhood $W$, and every 
	standard matching $M$, any color appears at most $\beta s$
	times in lists of edges of $M$ contained in $W$, 
	we have that the maximum number of conflict edges in 
	a subset of a given standard matching contained in a $6$-neighborhood
	is $\beta s$. 
	Since $\gamma \geq \beta$, this means that
	all permutations satisfy condition $(a)$ or $A=0$.
		
	\item  To estimate $B$, let $u$ be a fixed vertex of $G$, and
	let $P$ be a set of size $\gamma s$ ($|P|=\gamma s$) 
	of edges from $E_u$.
	For a vertex $v$ adjacent to $u$, if $uv$ is a conflict edge,
	then the colors used in $h$
	should be permuted in such a way that in the resulting coloring $h'$,
	the color of $uv$ is in $L(uv)$.
	Since $|L(uv)| \leq \beta s$, there are at most $(\beta s)^{\gamma s}$
	ways to choose which colors from $\{1,2,\dots,d\}$
	to assign to the edges in $P$ so that all
	edges in $P$ are conflict. The rest of the colors can be arranged in any 
	of the $(d - \gamma s)!$ possible ways. In total this gives at most
	$${d \choose \gamma s}(\beta s)^{\gamma s}(d-\gamma s)! 
	= \dfrac{d!(\beta s)^{\gamma s}}{(\gamma s)!}$$
	permutations that do not satisfy condition $(b)$ on vertex $u$.
	There are $n$ vertices in $G$, so we have
	$$B \leq n \dfrac{d!(\beta s)^{\gamma s}}{(\gamma s)!}$$
	
	\item To estimate $C$, let $uv$ be a fixed edge of $G$. 
	Each $2$-colored $4$-cycle $\mathcal{C} = uvztu$ 
	containing $uv$ is uniquely defined by an edge 
	$zt$ which is parallel with $uv$. Moreover, a permutation $\varsigma$ 
	contributes to $C$ if and only if there are at least $\tau s$ choices 
	for $zt$ so that 
	$\mathcal{C}$ is not allowed. We shall count the number of 
	ways $\varsigma$ could be 
	constructed for this to happen. First, note that for each 
	choice of a color $c_1$ from $\{1,\dots, d\}$, for the
	standard matching which contains $uv$, there are up 
	to $2\beta s$ cycles that
	are not allowed because of this choice. This follows from the fact
	that there are at most $\beta s$ choices
	for $t$ (or $z$) such that $L(ut)$ (or $L(vz)$) contains $c_1$. 
	So for a permutation 
	$\varsigma$ to contribute to $C$, $\varsigma$ must satisfy that 
	at least $(\tau-2\beta)s$ cycles
	containing $uv$ are forbidden because of 
	the color assigned to the standard matching containing $ut$ and
	$vz$.
	
	Let $S$ be a set of edges, $|S|=(\tau-2\beta) s$, such that for
	every edge $zt \in S$,
	the $2$-colored $4$-cycle $\mathcal{C}= uvztu$ is not allowed because
	of the color assigned to $ut$ and $vz$. 
	There are ${s-1} \choose {(\tau - 2\beta)s}$ ways to choose $S$.
	Furthermore, $L(uv)$ and $L(zt)$ contain at most $\beta s$ colors
	each, so there are at most
	$2\beta s$ choices for a color for the standard matching
	containing $ut$ and $vz$ 
	that would make $\mathcal{C}$ disallowed because
	of the color assigned to this standard matching.
	The remaining colors can be permuted in 
	$(d-1-(\tau-2\beta)s)!$ ways. 
	Thus, the total number of permutations $\sigma$ with not 
	enough allowed cycles for 
	a given edge is bounded from above by
	$$d{{s-1} \choose {(\tau - 2\beta)s}}{(2\beta s)^{(\tau - 2\beta)s}}
	(d-1-(\tau-2\beta)s)!$$
	
	Since $s \leq d$, this number is at most
	$$d{{d-1} \choose {(\tau - 2\beta)s}}{(2\beta s)^{(\tau - 2\beta)s}}
	(d-1-(\tau-2\beta)s)! $$
	Notice that the $d$-regular graph $G$ of order $n$ has $\dfrac{nd}{2}$ edges, so
	the total number of permutations $\sigma$ that have 
	too few allowed cycles for at 
	least one edge is bounded from above by
	$$C \leq \dfrac{nd}{2} d{{d-1} \choose {(\tau - 2\beta)s}}{(2\beta s)^{(\tau - 2\beta)s}}(d-1-(\tau-2\beta)s)! 
	= \dfrac{nd}{2}\dfrac{d!{(2\beta s)^{(\tau - 2\beta)s}}}{((\tau - 2\beta)s)!}$$
	\end{itemize}

	Hence,
	$$X \geq d! - n \dfrac{d!(\beta s)^{\gamma s}}{(\gamma s)!}
	- \dfrac{nd}{2}\dfrac{d!{(2\beta s)^{(\tau - 2\beta)s}}}{((\tau - 2\beta)s)!}$$
	
	Using Stirling's approximation $x! \geq x^xe^{-x}$, we have
	$$X \geq d!\Big(1 - n \dfrac{e^{\gamma s}(\beta s)^{\gamma s}}{(\gamma s)^{\gamma s}}
	- \dfrac{nd}{2}\dfrac{e^{(\tau - 2\beta)s}{(2\beta s)^{(\tau - 2\beta)s}}}
	{((\tau - 2\beta)s)^{(\tau - 2\beta)s}}\Big)$$
	$$X \geq d!\Big(1 - n\big(\dfrac{e \beta}{\gamma}\big)^{\gamma s}
	- \dfrac{nd}{2}\big(\dfrac{2e \beta}{\tau-2\beta}\big)^{(\tau-2\beta) s} \Big)$$
	
	Using the conditions 
	$$n\big(\dfrac{e \beta}{\gamma}\big)^{\gamma s} 
	 +\dfrac{nd}{2}\big(\dfrac{2e \beta}{\tau-2\beta}\big)^{(\tau-2\beta) s} < 1$$	
	we now deduce that $X>0$.

	\end{proof}

\begin{lemma}
\label{lemma2}
Let $h'$ be the proper $d$-edge coloring satisfying conditions $(a), (b), (c)$ of Lemma \ref{lemma1} and
$0<\gamma, \tau, \epsilon<1$ be parameters such that 
 \begin{equation}
\label{e2}
s -\tau s - 9\gamma s -  3\epsilon s - \dfrac{20\gamma}{\epsilon}d - 3 > 0
 \end{equation}
By performing a sequence of swaps on disjoint allowed $2$-colored $4$-cycles in $h'$, 
we obtain a proper $d$-edge coloring $h''$ of $G$ which avoids $L$.
\end{lemma}
\begin{proof}
For constructing $h''$ from $h'$, we will perform a number of swaps on $G$,
and we shall refer to this procedure as \textit{P-swap}. We are going to construct a set $P$ 
of disjoint allowed $2$-colored $4$-cycles such that each conflict of $h'$ with $L$ belongs
to one of them. An edge that belongs to a $2$-colored $4$-cycle 
in $P$ is called \textit{used} in $P$-swap. 
Suppose we have included a $4$-cycle $\mathcal{C}$ in $P$. Since for every $6$-neighborhood $W$ 
in $G$, and every standard matching $M$, the number of conflict edges in 
$M \cap E(W)$ is not greater than $\gamma s$, 
for every $5$-neighborhood $W$ in $G$,
the total number of edges in $W$ that are
used in $P$-swap is at most $4\gamma d s$.
A vertex $u$ in $G$ is {\em $P$-overloaded} if $E_u$ contains at least $\epsilon s$
edges that are used in $P$-swap; note that each used edge is incident to two vertices, thus no more than
$\dfrac{2 \times 4\gamma d s}{\epsilon s}=\dfrac{8\gamma}{\epsilon}d$ vertices of each $4$-neighborhood 
are $P$-overloaded. A standard matching $M$ in $G$ is {\em $P$-overloaded} in a 
$t$-neighborhood $W$ if $M \cap E(W)$ contains at least $\epsilon s$ edges that are used in $P$-swap; 
note that for each $5$-neighborhood $W$, no more than $\dfrac{4\gamma}{\epsilon}d$ standard matchings
of $G$ are $P$-overloaded in $W$.

Using these facts, let us now construct our set $P$ by steps; at each step
we consider a conflict edge $e$ and include an allowed $2$-colored $4$-cycle 
containing $e$ in $P$. Initially, the set $P$ is empty. Next, for each conflict edge 
$e=uv$ in $G$, there are at least $s- \tau s$ allowed cycles 
containing $e$. We choose an allowed cycle $uvztu$ which contains $e$ and satisfies the following:
\begin{itemize}
	\item[(1)] $z$ and $t$ and the standard matching that 
	contains $vz$ and $ut$ 
	are not $P$-overloaded in the $4$-neighborhood $W_e$
	of $e$; this eliminates at most $\dfrac{2 \times 8\gamma}{\epsilon}d + \dfrac{4\gamma}{\epsilon}d=\dfrac{20\gamma}{\epsilon}d$ choices.
        Note that with this strategy for including 
	$4$-cycles in $P$, after completing the
	construction of $P$, every vertex is incident 
	with at most $2\gamma s+(\epsilon s -1) +2 = 2\gamma s+\epsilon s +1$
	edges that are used in $P$-swap. Furthermore, after we have 
	constructed the set $P$, no standard matching 
	contains more than $2\gamma s+\epsilon s+1$ edges that are used in $S$-swap
	in a $1$-neighborhood of $G$ ; this
	follows from the fact that every $1$-neighborhood $W'$ in $G$  that $ut$, $vz$ or $zt$
	belongs to is contained in $W_e$. 
	
	\item[(2)] None of the edges $vz$, $zt$, $ut$ are 
	conflict, or used before in $P$-swap. 
	
	All possible choices for these edges
	are in the $1$-neighborhood $W_e$ of $e$ in $G$ .
	Since no vertex  in $W_e$ or 
	subset of a standard matching that is in $W_e$
	contains more than $\gamma s$
	conflict edges and $P$-swap
	uses at most $2\gamma s+\epsilon s +1$ edges
	at each vertex and in each subset of a standard matching
	contained in $W_e$, these restrictions
	eliminate at most 
	$3 \gamma s + 3(2 \gamma s+\epsilon s +1)$ 
	or $9 \gamma s + 3\epsilon s+3$  choices.
\end{itemize}
	It follows that we have at least
	$$s -\tau s - 9\gamma s -  3\epsilon s - \dfrac{20\gamma}{\epsilon}d - 3$$
	choices for an allowed cycle $uvztu$ which contains $uv$.
	By assumption, this expression is greater than zero,
	so we conclude that there is a cycle satisfying these conditions,
	and thus we may construct the set $P$ by iteratively adding disjoint allowed
	$2$-colored $4$-cycles such that each cycle contains a conflict edge.
	After this process terminates we have a set $P$ of disjoint 
	allowed cycles; we swap on all the cycles in $P$ to obtain the coloring $h''$ which avoids $L$.
\end{proof}

We now complete the proof of the main Theorem \ref{maintheorem}:
\begin{proof} 
By assumption, we have $\beta \leq 2^{-11}sd^{-1} (2n)^{-2^9ds^{-2}}$. 
Since $s \leq d$ and $(2n)^{-2^9ds^{-2}} \leq 1$, it follows that $\beta \leq 2^{-11} sd^{-1} \leq 2^{-11}$.
Let $\gamma=2^{-9}sd^{-1}$; then $\beta \leq \gamma$ and
\begin{equation}
\label{e11}
n\big(\dfrac{e \beta}{\gamma}\big)^{\gamma s} 
< n \big(\dfrac{2^2 2^{-11}sd^{-1} (2n)^{-2^{9}ds^{-2}}}{2^{-9}sd^{-1}}\big)^{2^{-9}sd^{-1} s}=n(2n)^{-1}=\dfrac{1}{2}.
\end{equation}
Let $\tau=2^{-7}$, then $\tau -2\beta \geq 2^{-7} -2.2^{-11}>2^{-8}$ and
$$\dfrac{\tau -2\beta}{2e\beta}>\dfrac{2^{-8}}{2e2^{-11}sd^{-1} (2n)^{-2^9ds^{-2}}}
>\dfrac{1}{(2n)^{-2^9ds^{-2}}}
>(2n)^{2^9ds^{-2}}>1$$
This implies
$$\big(\dfrac{\tau -2\beta}{2e\beta}\big)^{(\tau-2\beta) s}>\big(\dfrac{\tau -2\beta}{2e\beta}\big)^{2^{-8}s}
>\big((2n)^{2^9ds^{-2}}\big)^{2^{-8}s}=(2n)^{2ds^{-1}}>(2n)^2$$
Using the fact that $d <n$, we have
\begin{equation}
\label{e12}
\dfrac{nd}{2}\big(\dfrac{2e \beta}{\tau-2\beta}\big)^{(\tau-2\beta) s} 
<\dfrac{nd}{2} \dfrac{1}{(2n)^2} < \dfrac{1}{8}.
\end{equation}
Combining (\ref{e11}) and (\ref{e12}), we obtain
$n\big(\dfrac{e \beta}{\gamma}\big)^{\gamma s} +\dfrac{nd}{2}\big(\dfrac{2e \beta}{\tau-2\beta}\big)^{(\tau-2\beta) s} < 1$.
Since the values of $\gamma, \tau$ satisfy the conditions in Lemma \ref{lemma1}, there is a permutation $\rho$ of the colors
in the standard $d$-edge coloring $h$ of 
$G$ from which we obtain a proper $d$-edge coloring $h'$ of $G$ satisfying 
the conditions $(a), (b), (c)$ in Lemma \ref{lemma1}. 
 Furthermore, since $d\geq s \geq 11$, if we let $\epsilon=2^{-3}$, then we have
$$s -\tau s - 9\gamma s -  3\epsilon s - \dfrac{20\gamma}{\epsilon}d - 3>0.$$
It follows that Lemma \ref{lemma2} yields a proper $d$-edge coloring $h''$ of $G$ which avoids $L$.
\end{proof}


\section{Families of $(d,s)$-edge colorable graphs}
The results in \cite{Lina} and \cite{JohanQd} allow us to conclude that
the complete bipartite graph $K_{d,d}$ ($d=2^t, t \in \mathbb{N}$) and 
the hypercube graph $Q_d$ ($d \in \mathbb{N}$) are $(d,d)$-edge colorable graphs.
Compared to the hypercube graph
$Q_d$ ($d \in \mathbb{N}$), the complete bipartite graph $K_{d,d}$ ($d=2^t, t \in \mathbb{N}$) is much denser, so it is interesting 
to see how this complete bipartite graph behaves if some edges are removed.
Lemma \ref{Completedelete} considers the case when we take the set of
$k$ standard matchings out of $K_{d,d}$ ($d=2^t, t \in \mathbb{N}$).

\begin{lemma}
\label{Completedelete}
The graph $G$ obtained by removing $k$ standard matchings from
the complete bipartite graph $K_{d,d}$ $(d=2^t, t \in \mathbb{N})$ 
is a $(d-k, d-k)$-edge colorable graphs.
\end{lemma}
\begin{proof}
It is straightforward that $G$ is a $(d-k)$-regular graph. Let $h$ be a standard coloring of
$K_{d,d}$ such that all edges in a standard matching of $K_{d,d}$ receive the same color in $h$.
We define the proper $(d-k)$-edge coloring $h'$ of $G$ from $h$
by retaining the color of every non-deleted edge in $E(G)$.
Consider an edge $e$ of $G$, 
removing one standard matching from $K_{d,d}$ eliminates one 
$2$-colored $4$-cycle that contains $e$. Hence, the number of $2$-colored $4$-cycles containing $e$ 
is $d-1-k$, this implies $G$ is a $(d-k, d-k)$-edge colorable graphs.
\end{proof}

Recall that the Cartesian product $G=G_1 \square \, G_2$ 
of the graphs $G_1$ and $G_2$ is a graph whose
vertex set 
is the Cartesian product $V(G_1) \times V(G_2)$
and where two vertices $u=(u_1,u_2)$ and $v=(v_1,v_2)$ are adjacent in $G$ 
whenever $u_1=v_1$ and $u_2$ is adjacent with $v_2$ in $G_2$, or $u_2=v_2$ and $u_1$ is adjacent with $v_1$ in $G_1$.
The following lemma yields $(d,s)$-edge colorable graphs by taking
Cartesian product of graphs.

\begin{lemma}
\label{Cartesianproduct}
Let $G_1$ be a $(d_1,s_1)$-edge colorable graph and $G_2$ be a $(d_2,s_2)$-edge colorable graph.
Then the Cartesian product $G=G_1 \square \, G_2$ of graphs $G_1$ and $G_2$ is a $(d, s)$-edge colorable graph
with $d=d_1+d_2$ and $s=\min \{d_1 +s_2 ,d_2+s_1\}$.
\end{lemma}
\begin{proof}
By the definition of the Cartesian product of graphs, it is straightforward that $G$ is $d$-regular graph with $d=d_1+d_2$.
Let $h_1$ be a standard coloring of $G_1$ and $h_2$ be a standard coloring of $G_2$
such that the set of colors in $h_1$ and the set of colors in $h_2$ are disjoint.
We define an edge coloring $h$ of $G$:
For two adjacent vertices $u=(u_1,u_2)$ and $v=(u_1,v_2)$ in $G$, the edge $uv$ is given the color $h(uv)=h_2(u_2v_2)$
and for two adjacent vertices $u=(u_1,u_2)$ and $v=(v_1,u_2)$ in $G$, the edge $uv$ is given the color $h(uv)=h_1(u_1v_1)$.
Thus $h$ is a proper $d$-edge coloring of $G$.

Note that an edge $uv$ of $G$ with $u=(u_1,u_2)$ and $v=(u_1,v_2)$ is contained in a $2$-colored $4$-cycle
$uvztu$ with $z=(u_i,v_2)$ and $t=(u_i,u_2)$ ($u_i$ is neighbour of $u_1$ in $G_1$). Furthermore, if $u_2v_2z_2t_2u_2$ is
a $2$-colored $4$-cycle in $G_2$, then $uvz't'u$ with $z'=(u_1,z_2)$ and $t'=(u_1,t_2)$ is a $2$-colored $4$-cycle in $G$.
Since the degree of $u_1$ is $d_1$ and every edge of $G_2$ is in at least $s_2 -1$ $2$-colored $4$-cycles,
an edge $uv$ of $G$ with $u=(u_1,u_2)$ and $v=(u_1,v_2)$ belongs to at least $d_1 + s_2 -1$ $2$-colored $4$-cycles. 
Similarly, an edge $uv$ of $G$ with $u=(u_1,u_2)$ and $v=(v_1,u_2)$ belongs to  at least $d_2 +s_1 -1$ $2$-colored $4$-cycles. 
Therefore, we can conclude that $G$ is a $(d, s)$-edge colorable graph with $d=d_1+d_2$ and $s=\min \{d_1 +s_2 ,d_2+s_1\}$.

\end{proof}

In the remaining part of this section, we examine some other
graphs that belong to the family of $(d,s)$-edge colorable graphs.

Let $G$ be a finite group and let $S$ be a generating set of $G$ such that $S$ does not contain
the identity element $e$, $|S|=d$ and $S=S^{-1}$ (which means if $a \in S$ then $a^{-1} \in S$). 
The undirected Cayley graph Cay($G$, $S$) over the set 
$S$ is defined as the graph whose vertex set is $G$ and 
where two vertices $a,b \in G$ are adjacent 
whenever $\{ab^{-1}, ba^{-1}\} \subseteq S$. It is straightforward that Cay($G$, $S$)
is a $d$-regular graph, Lemma \ref{Cayley1} and Lemma \ref{Cayley2}
show that if $S$ satisfies some further conditions then Cay($G$, $S$) is a $(d,s)$-edge colorable graph.

\begin{lemma}
\label{Cayley1}
Let Cay\emph{($G$, $S$)} be an undirected Cayley graph  on a 
group $G$ over the generating set 
$S \subseteq G \setminus \{e\}$. If $a=a^{-1}$ for every $a \in S$, $|S|=d$ and there exits a subset $S_c \subseteq  S$, 
$|S_c|=s$, satisfying that every element of $S_c$ is commutative with all elements in $S$, then Cay\emph{($G$, $S$)} is a $(d,s)$-edge colorable graph.
\end{lemma}
\begin{proof}
Let $h$ be the proper $d$-edge coloring of Cay($G$, $S$) such that every edge $uv$ in Cay($G$, $S$)
is colored $a$ if $uv^{-1}=vu^{-1}=a \in S$. For an edge $uv$ colored $a$, consider an arbitrary element
$b \in S_c$ and let $z=vb$, $t=ub$, then the edges $vz$ and $ut$ are colored $b$ in $h$. Furthermore, since $b$ 
is commutative with $a$, i.e. $ab=ba$, we have $z=vb=uab=uba=ta$. This implies that there is an edge between 
$z$ and $t$, and this edge is colored $a$ in $h$. Hence, $uvztu$ is a $2$-colored $4$-cycle. 
Because $|S_c|=s$, each edge of Cay($G$, $S$) is in at least $s-1$ $2$-colored $4$-cycles.
It follows that Cay($G$, $S$) is a $(d,s)$-edge colorable graph.
\end{proof}

\begin{lemma}
\label{Cayley2}
Let Cay\emph{($G$, $S$)} be an undirected Cayley graph on an
Abelian group $G$ over the generating set 
$S \subseteq G \setminus \{e\}$, $S=S^{-1}$ and $|S|=d$. Let 
$S_k=\{s_1, s_2,...,s_k\}$ be a subset
of $S$ such that $S_k \cup S^{-1}_k=S$ and $S_k$ does not contain two different elements $s_i \neq s_j$ satisfying 
that $s_i=s^{-1}_j$.
If $S$ has the following properties:
\begin{itemize}
\item[(i)] every element $s_i$ of $S_k$ has even order $d_i$ $({s_i}^{d_i}=s_i^0=e)$;
\item[(ii)] for every element $g \in G$, there is exactly one 
sequence $(x_1,x_2,\dots,x_k)$ $(x_i \in [0,d_i -1]$ for $i \in [1,k])$ such that $g={s_1}^{x_1}{s_2}^{x_2}...{s_k}^{x_k}$;
\end{itemize}
then Cay\emph{($G$,$S$)} is a $(d,d)$-edge colorable graph.
\end{lemma}

\begin{proof}
We write $S=\{s_1,\dots,s_k, s^{-1}_1,\dots,s^{-1}_k\}$; note that the size of $S$ may not be $2k$, since
$S$ may contain some  element $s_x$ with $s_x=s^{-1}_x$. Consider an edge $uv$ of Cay($G$, $S$);
without loss of generality assume that $v=us_i$ ($u=vs^{-1}_i$)  for some $i \in [1,k]$.
The condition $(ii)$ implies that there exists exactly one sequence
$(x_1,x_2,\dots,x_k)$ such that  $u={s_1}^{x_1} \dots {s_i}^{x_i} \dots {s_k}^{x_k}$
and $v={s_1}^{x_1}\dots{s_i}^{(x_i+1) \mod {d_i}}\dots{s_k}^{x_k}$.
We color the edge $uv$ by color $s_i$ if $x_i$ is even, and by color $s^{-1}_i$ if $x_i$ is odd.
By repeating this for all edges of Cay($G$, $S$), we obtain the proper $d$-edge 
coloring $h$.

Given an edge $e$ of Cay($G$, $S$), let $u$ and $v$ be the two endpoints
of $e$, where $v=us_i$ (for some $i \in [1,k]$).
Let $u={s_1}^{x_1}...{s_i}^{x_i}...{s_k}^{x_k}$; $x_i$
is called the {\em power of $s_i$ in $u$} ($i \in [1,k]$) 
and denoted by $p_u(s_i)$.
Consider an arbitrary element $s \in S$,
if $s=s_j \in S_k$, let $z=vs_j$, $t=us_j$, then 
$h(vz)=h(ut)=s_j$ or $h(vz)=h(ut)=s^{-1}_j$ since $p_u(s_j)= p_v(s_j)$
and $p_t(s_j)= p_z(s_j)=p_u(s_j)+1$. Furthermore, since $G$ is an Abelian group,
we have $z=vs_j=us_is_j=us_js_i=ts_i$. Thus there is an edge between 
$z$ and $t$, and $h(uv)=h(tz)$ since
$p_u(s_i)= p_t(s_i)$ and $p_v(s_i)= p_z(s_i)=p_u(s_i)+1$; hence 
$uvztu$ is a $2$-colored $4$-cycle. 
If $s \in S^{-1}_k$, we proceed similarly.
Because $|S|=d$, each edge of Cay($G$,$S$) is in at least $d-1$ 
$2$-colored $4$-cycles. It follows that Cay($G$, $S$) is a $(d,d)$-edge colorable graph.
\end{proof}
\section{Acknowledgement}
The author would like to thank Klas Markstr\"om and Carl Johan Casselgren for their help and useful discussion.
						
\end{document}